\documentclass[10pt]{amsart}

\usepackage{amssymb}
\usepackage{amsthm}
\usepackage[hidelinks]{hyperref}
\usepackage{mathtools}
\usepackage{multirow}
\usepackage{bbm}

\newtheorem{n}{}[section]

\theoremstyle{plain}
\newtheorem{corollary}[n]{Corollary}
\newtheorem{proposition}[n]{Proposition}
\newtheorem{theorem}[n]{Theorem}

\theoremstyle{definition}
\newtheorem{definition}[n]{Definition}
\newtheorem{algorithm}[n]{Algorithm}
\newtheorem{example}[n]{Example}
\newtheorem{notation}[n]{Notation}
\newtheorem{remark}[n]{Remark}

\newcommand{\1}{\mathbbm{1}}

\newcommand{\A}{\mathbb{A}}
\newcommand{\C}{\mathbb{C}}
\newcommand{\F}{\mathbb{F}}
\newcommand{\N}{\mathbb{N}}
\newcommand{\Q}{\mathbb{Q}}

\newcommand{\Z}{\mathbb{Z}}

\newcommand{\EE}{\mathcal{E}}
\newcommand{\FF}{\mathcal{F}}
\newcommand{\LL}{\mathcal{L}}
\newcommand{\OO}{\mathcal{O}}

\newcommand{\ff}{\mathfrak{f}}

\newcommand{\Alb}{\operatorname{Alb}}
\newcommand{\cc}{\operatorname{cc}}
\newcommand{\ch}{\operatorname{char}}
\renewcommand{\d}{\operatorname{d}\!}
\newcommand{\et}{\text{\'et}}
\newcommand{\Gal}{\operatorname{Gal}}
\newcommand{\GL}{\operatorname{GL}}
\newcommand{\id}{\operatorname{id}}
\newcommand{\Ind}{\operatorname{Ind}}
\newcommand{\NS}{\operatorname{NS}}
\newcommand{\ord}{\operatorname{ord}}
\newcommand{\tr}{\operatorname{tr}}

\title[Computing L-functions of global function fields]{Computing L-functions of $ \lambda $-adic representations of global function fields}
\author{David Kurniadi Angdinata}

\begin{document}

\begin{abstract}
The L-function $ L(\rho_\lambda, s) $ of an almost everywhere unramified $ \lambda $-adic representation $ \rho_\lambda $ of a global function field $ \F_q(C) $ is known to be a rational function in $ q^{-s} $ satisfying a functional equation up to some complex sign $ \epsilon(\rho_\lambda) $. This paper presents a systematic framework to compute the coefficients of $ L(\rho_\lambda, s) $ and its sign $ \epsilon(\rho_\lambda) $ with some explicit examples.


\end{abstract}

\maketitle

\section{Introduction}

Given a global field, there are L-functions $ L(\rho, s) $ in the complex variable $ s \in \C $ associated to suitable continuous representations $ \rho $ of its absolute Galois group, which are conjecturally meromophic over $ \C $ and conjecturally satisfy functional equations. These L-functions play the role of half of the modern Langlands philosophy, and there are many crucial algebraic and analytic questions that remain unanswered. For instance, when $ \rho = \chi $ is a Dirichlet character, the generalised Riemann hypothesis describes where its $ \zeta $-function $ \zeta(s) = L(\chi, s) $ vanishes, and when $ \rho = \rho_{A, \ell} $ arises from an abelian variety $ A $, the Birch and Swinnerton-Dyer conjecture describes the arithmetic of $ A $ in terms of its L-function $ L(A, s) = L(\rho_{A, \ell}, s) $. The importance of L-functions in arithmetic algebraic geometry is so much so that thousands of CPU years have been spent on computing their data for various representations in the massive L-functions and modular forms database (LMFDB) \cite{LMFDB}, which has proven to be an immensely useful resource for number theorists.

At a cursory glance, most of the data in the LMFDB are of objects over various number fields. In this setting, their L-functions are transcendental functions of $ s $ that can be written as Dirichlet series, with the aforementioned conjectures still open for most objects. Much of their data were computed from Platt's code for degree one L-functions \cite{Pla11} based on Booker's algorithm \cite{Boo06}, from the \texttt{lcalc} library in C++ based on Rubinstein's algorithm \cite{Rub05} for Maass forms, and from the work of Booker--Sijsling--Sutherland--Voight--Yasaki \cite{BSSVY16} for genus two curves. In a much larger generality, T Dokchitser wrote an efficient algorithm to compute the special value of a motivic L-function, assuming a conjectural functional equation that can be numerically verified \cite{Dok04}. This was originally implemented as the PARI \cite{PARI} package \texttt{ComputeL}, but has since been ported to the core libraries of Magma \cite{BCP97} as the \texttt{LSeries} function and SageMath \cite{SageMath} as the \texttt{dokchitser} function.

In a stark contrast, the arithmetic of analogous objects over $ K \coloneqq \F_q(C) $ is very well-understood \cite{Laf02}. In this setting, their L-functions are rational functions of $ q^{-s} $, satisfying a known functional equation and even the Riemann hypothesis. Yet, the computational ecosystem over $ K $ is severely limited, with little to no support for explicit computations in the core libraries of most computer algebra systems, including PARI and SageMath. Magma is better in this regard, including multiple constructors for $ K $, explicit class field theory computations with \texttt{WittRing} and \texttt{CarlitzModule}, and even an \texttt{LFunction} for non-constant elliptic curves over $ \F_q(t) $. Using this interface, Comeau--Lapointe--David--Lal\'in--Li computed $ L(E, \chi, s) = L(\rho_{E, \ell} \otimes \chi, s) $ for a fixed Legendre curve $ E $ over $ K = \F_q(t) $ twisted by Dirichlet characters $ \chi $ to investigate their vanishing \cite{CLDLL22}. In private communication, Maistret--Wiersema also computed $ L(E, \chi, s) $ for many constant elliptic curves $ E $ by hand to formulate an analogue over $ K $ of the twisted BSD-type formula of V Dokchitser--Evans--Wiersema \cite{DEW21}. This seems to be the extent to which L-functions of $ K $ have been computed explicitly in the literature, which renders the creation of an analogue of the LMFDB over $ K $ difficult.

This paper presents an analogue of T Dokchitser's algorithm for $ K $, in the language of \emph{$ \lambda $-adic representations $ \rho_\lambda $} for some prime $ \lambda \nmid q $ of a fixed number field $ F $. These are the natural Galois representations associated to L-functions of motivic origin, including Riemann/Dedekind $ \zeta $-functions, Dirichlet/Weber/Hecke/Artin L-functions, and Hasse--Weil L-functions of varieties. Then there is a rational function $ \LL(\rho_\lambda, T) \in F_\lambda(T) $, called the \emph{formal L-function} of $ \rho_\lambda $, such that $ \LL(\rho_\lambda, q^{-s}) = L(\rho_\lambda, s) $. Under the condition that $ \rho_\lambda $ is \emph{self-dual of weight $ w(\rho_\lambda) \in \N $ and sign $ c(\rho_\lambda) \in \{\id, \cc\} $}, where $ \id $ is the identity and $ \cc $ is complex conjugation, there is a field element $ \epsilon(\rho_\lambda) \in F_\lambda^\times $ such that $ \LL(\rho_\lambda, T) $ satisfies a functional equation
$$ \LL(\rho_\lambda, T) = \epsilon(\rho_\lambda) \cdot T^{n(\rho_\lambda) - d(\rho_\lambda)} \cdot \LL(\rho_\lambda, (q^{w(\rho_\lambda) + 1}T)^{-1})^{c(\rho_\lambda)}, $$
where $ n(\rho_\lambda), d(\rho_\lambda) \in \N $ are the degrees of the numerator and denominator of $ \LL(\rho_\lambda, T) $ respectively. If the genus of $ C $ and the dimension of $ \rho_\lambda $ are small, then $ n(\rho_\lambda) - d(\rho_\lambda) $ is essentially the degree of the \emph{Artin conductor} of $ \rho_\lambda $.

\begin{theorem}[Corollary \ref{cor:rationalfunction}, Corollary \ref{cor:functionalequation}, Corollary \ref{cor:computeepsilon}]
\label{thm:mainresults}
Let $ \rho_\lambda $ be an almost everywhere unramified $ \lambda $-adic representation of a global function field $ K $.
\begin{itemize}
\item[(1)] There is an algorithm to compute $ \LL(\rho_\lambda, T) $ with a running time that is roughly exponential in $ n(\rho_\lambda)$.
\end{itemize}
Assume further that $ \rho_\lambda $ is self-dual of weight $ w(\rho_\lambda) $ and sign $ c(\rho_\lambda) \in \{\id, \cc\} $.
\begin{itemize}
\item[(2)] If $ \epsilon(\rho_\lambda) $ is known, then there is an algorithm to compute $ \LL(\rho_\lambda, T) $ with a running time that is roughly exponential in $ \lfloor n(\rho_\lambda) / 2\rfloor $.
\item[(3)] Otherwise, there is an algorithm to compute $ \epsilon(\rho_\lambda) $ with a running time that is typically exponential in $ \lfloor n(\rho_\lambda) / 2\rfloor $, and at worst exponential in $ n(\rho_\lambda) $.
\end{itemize}
\end{theorem}

The algorithms in Theorem \ref{thm:mainresults} have been implemented in Magma \cite{Ang25} and bug-tested against existing implementations of $ L(E, s) $ and $ L(E, \chi, s) $. The proof of (1) is elementary and only uses properties of formal power series. The proofs of (2) and (3) use a generalisation of (1), and are heavily inspired by the algorithm of Comeau-Lapointe--David--Lal\'in--Li for $ L(E, \chi, s) $ \cite[Section 5.1]{CLDLL22}, although their analogue of (1) is an ad-hoc computation using Legendre symbols.

Section \ref{sec:adicrepresentations} reviews some useful background on invariants associated to $ \rho_\lambda $, including $ \LL(\rho_\lambda, T) $ and $ \epsilon(\rho_\lambda) $, which are sometimes implicit in the literature without good references. Section \ref{sec:computingfunctions} details Theorem \ref{thm:mainresults} and briefly justifies their running time complexities. Section \ref{sec:examplecomputations} works them out by hand for explicit examples of elliptic curves, Dirichlet characters, and other motivic objects.

\pagebreak

\section{\texorpdfstring{$ \lambda $}{l}-adic representations}
\label{sec:adicrepresentations}

In this section, $ K $ will be a non-archimedean local field or a global function field, which will be clear from context, of residue size $ q \in \N $. Moreover, let $ \lambda \nmid q $ be a fixed prime of a coefficient number field $ F $, and let $ V $ and $ V_\lambda $ be fixed finite-dimensional vector spaces over $ F $ and $ F_\lambda $ respectively. If $ F = \Q $, write $ \lambda = \ell $ instead.

\subsection{Local invariants}

In this subsection, let $ K $ be a non-archimedean local field, let $ \OO_K $ be its valuation ring with normalised valuation $ |\cdot|_K $, normalised uniformiser $ \varpi_K $, and residue field size $ q_K $ such that $ |\varpi_K|_K = q_K^{-1} $, and let $ W_K $ be its Weil group with inertia subgroup $ I_K $ and a fixed choice of geometric Frobenius $ \phi_K $.

A \textbf{$ \lambda $-adic representation} of $ K $ is a continuous homomorphism $ \rho_\lambda : W_K \to \GL(V_\lambda) $, where $ \GL(V_\lambda) $ is endowed with the $ \lambda $-adic topology.

\begin{example}
A smooth proper variety $ X $ over $ K $ has a natural $ \lambda $-adic representation $ \rho_{X, \ell} $ given by the continuous action of $ W_K $ on the first $ \ell $-adic cohomology group $ H_\et^1(\overline{X}, \Q_\ell) $ of $ X $ induced by functoriality. Then $ \rho_{X, \ell} $ is \emph{independent of $ \ell $}, in the weak sense that the characteristic polynomial of $ \rho_{X, \ell}(\phi_K) $ has coefficients in $ \Q $, since it is true when $ X $ is abelian \cite[Theorem 4.3]{GR72} and $ H_\et^1(\overline{X}, \Q_\ell) \cong H_\et^1(\Alb(\overline{X}), \Q_\ell) $ \cite[Theorem 2A9(6)]{Kle68}, where $ \Alb(X) $ is the Albanese abelian variety of $ X $.
\end{example}

\begin{example}
An \emph{Artin representation} of $ K $ is a continuous homomorphism $ \varrho : W_K \to \GL(V) $, where $ \GL(V) $ is endowed with the discrete topology. This can be viewed as a $ \lambda $-adic representation of $ K $ via a fixed embedding $ V \hookrightarrow V_\lambda $.
\end{example}

The invariants of $ \rho_\lambda $ are best defined in terms of \emph{Weil--Deligne representations} $ (\rho, \nu) $. This is the data of a \emph{Weil representation} $ \rho $, which is a continuous homomorphism from $ W_K $ to $ \GL(V') $, where $ V' $ is a finite-dimensional vector space over an arbitrary field $ F' $ with $ \ch(F') = 0 $ endowed with the discrete topology, and a nilpotent endomorphism $ \nu : F' \to F' $ such that for any $ c \in W_K $,
$$ \rho(c) \cdot \nu \cdot \rho(c)^{-1} = |c|_K \cdot \nu. $$
It turns out that there is an explicit equivalence between the symmetric monoidal categories of $ \lambda $-adic representations $ \rho_\lambda $ and of Weil--Deligne representations $ (\rho, \nu) $ \cite[Section 8]{Del73b}, where the tensor product is given by
$$ (\rho, \nu) \otimes (\sigma, \upsilon) \coloneqq (\rho \otimes \sigma, \nu \otimes 1 + 1 \otimes \upsilon). $$
Then $ \rho_\lambda $ is said to be \emph{unramified} if $ \rho $ is unramified and $ \nu = 0 $, and \emph{indecomposable} if $ (\rho, \nu) $ is not a direct sum of subspaces invariant under both $ \rho $ and $ \nu $.

For any subrepresentation $ \sigma : W_K \to \GL(V') $ of $ \rho $, let $ \sigma^G $ denote the subrepresentation of $ \sigma $ invariant under a subgroup $ G $ of $ W_K $, namely
$$ \sigma^G \coloneqq \{\xi \in V' : \forall c \in G, \ \xi^{\sigma(c)} = \xi\}. $$
The \textbf{Artin conductor exponent} $ \alpha(\rho_\lambda) $ of $ \rho_\lambda $ is the sum of
$$ \alpha(\rho) \coloneqq \sum_{r = 0}^\infty \dfrac{\dim\rho - \dim\rho^{G_{K, r}}}{[G_{K, r} : I_K]}, \qquad \beta(\rho_\lambda) \coloneqq \dim\rho^{I_K} - \dim\ker(\nu)^{I_K}, $$
where $ G_{K, r} $ is the $ r $-th higher ramification group in lower numbering of $ K $ for each $ r \in \N $, so that $ G_{K, 0} = I_K $. Then $ \alpha(\rho_\lambda) = 0 $ precisely when $ \rho_\lambda $ is unramified, since $ G_{K, r} = 0 $ for any $ r \in \N $ and $ \nu = 0 $. Note that this one of many equivalent definitions \cite[Section 4 to Section 8]{Ulm16}. When $ \rho_\lambda = \rho_{E, \ell} $ arises from an elliptic curve $ E $ over $ K $, this is equivalent to the usual conductor exponent of $ E $.

\pagebreak

Now let $ \psi $ be a non-trivial additive character of $ K $, let $ \mu $ be an additive Haar measure on $ K $, and let $ \gamma(\psi) $ be the smallest integer $ r \in \Z $ such that $ \psi(\varpi_K^r\OO_K) = 1 $. Langlands \cite{Lan70} and Deligne \cite{Del73b} independently gave non-constructive proofs on the existence of a constant $ \epsilon(\rho, \psi, \mu) \in \C^\times $ unique under the following properties.
\begin{itemize}
\item Additivity: if $ \sigma $ and $ \tau $ are Weil representations of $ K $ in a short exact sequence $ 0 \to \rho \to \sigma \to \tau \to 0 $, then
$$ \epsilon(\sigma, \psi, \mu) = \epsilon(\rho, \psi, \mu) \cdot \epsilon(\tau, \psi, \mu). $$
\item Inductivity in degree zero: if $ \rho' $ is a virtual Weil representation of degree zero over a finite separable extension $ K' $ of $ K $, and $ \mu' $ is an additive Haar measure on $ K' $, then
$$ \epsilon(\Ind_K^{K'}\rho, \psi, \mu) = \epsilon(\rho', \psi \circ \tr_{K' / K}, \mu'). $$
\item Quasi-character formula: if $ \rho $ is one-dimensional and corresponds to a quasi-character $ \chi : W_K \to \C^\times $ via local class field theory, then
$$ \epsilon(\rho, \psi, \mu) =
\begin{cases}
\chi(\phi_K)^{\gamma(\psi)} \cdot q_K^{\gamma(\psi)} \cdot \mu(\OO_K) & \text{if} \ \rho \ \text{is unramified}, \\
\sum_{r \in \Z} \int_{\varpi_K^r\OO_K} \chi^{-1}(x)\psi(x)\d\mu(x) & \text{otherwise}.
\end{cases}
$$
\end{itemize}
This depends on $ \psi $ and $ \mu $ in general, but Langlands and Deligne made specific choices whose differences in convention are summarised by Tate \cite[Section 3.6]{Tat79}. The \textbf{local $ \epsilon $-factor} of $ \rho_\lambda $ is then $ \epsilon(\rho_\lambda, \psi, \mu) \coloneqq \epsilon(\rho, \psi, \mu) \cdot \delta(\rho_\lambda) $, where
$$ \delta(\rho_\lambda) \coloneqq \dfrac{\det(-\rho^{I_K}(\phi_K))}{\det(-\ker(\nu)^{I_K}(\phi_K))}. $$
The following generalises a well-known formula for tensor products of Weil representations with one unramified factor to that of $ \lambda $-adic representations.

\begin{proposition}
\label{prop:unramifiedepsilon}
Let $ \rho_\lambda $ and $ \sigma_\lambda $ be $ \lambda $-adic representations of a non-archimedean local field $ K $ such that $ \sigma_\lambda $ is unramified. Then
$$ \epsilon(\rho_\lambda \otimes \sigma_\lambda, \psi, \mu) = \epsilon(\rho_\lambda, \psi, \mu)^{\dim\sigma_\lambda} \cdot \det\sigma_\lambda(\phi_K)^{\alpha(\rho_\lambda) + \gamma(\psi)\dim\rho_\lambda}. $$
\end{proposition}

\begin{proof}
Let $ (\rho, \nu) $ and $ (\sigma, \upsilon) $ be the Weil--Deligne representations associated to $ \rho_\lambda $ and $ \sigma_\lambda $ respectively. Then \cite[Equation 5.5.3]{Del73b}
$$ \epsilon(\rho \otimes \sigma, \psi, \mu) = \epsilon(\rho, \psi, \mu)^{\dim\sigma} \cdot \det\sigma(\phi_K)^{\alpha(\rho) + \gamma(\psi)\dim\rho}. $$
On the other hand,
$$
\begin{aligned}
\delta(\rho_\lambda \otimes \sigma_\lambda)
& = \dfrac{\det(-\rho^{I_K}(\phi_K) \otimes \sigma(\phi_K))}{\det(-\ker(\nu)^{I_K}(\phi_K) \otimes \sigma(\phi_K))} \\
& = \dfrac{\det(-\rho^{I_K}(\phi_K))^{\dim\sigma} \cdot \det\sigma(\phi_K)^{\dim\rho^{I_K}}}{\det(-\ker(\nu)^{I_K}(\phi_K))^{\dim\sigma} \cdot \det\sigma(\phi_K)^{\dim\ker(\nu)^{I_K}}} \\
& = \delta(\rho_\lambda)^{\dim\sigma} \cdot \det\sigma(\phi_K)^{\beta(\rho_\lambda)}.
\end{aligned}
$$
The result follows by multiplying both expressions.
\end{proof}

In particular, by Proposition \ref{prop:unramifiedepsilon} with $ \rho_\lambda = \1 $ and the quasi-character formula,
$$ \epsilon(\sigma_\lambda, \psi, \mu) = (q_K^{\gamma(\psi)} \cdot \mu(\OO_K))^{\dim\sigma_\lambda} \cdot \det\sigma_\lambda(\phi_K)^{\gamma(\psi)}, $$
so that $ \epsilon(\sigma_\lambda, \psi, \mu) = 1 $ whenever $ \gamma(\psi) = 0 $ and $ \mu(\OO_K) = 1 $.

\pagebreak

\subsection{Global invariants}

Now let $ K $ be a \emph{global function field}, namely the function field $ \F_q(C) $ of a smooth proper geometrically irreducible curve $ C $ of genus $ g(C) \in \N $ over the finite field $ \F_q $ of prime power size $ q \in \N $. The valuative criterion for properness gives a bijective correspondence between the set of closed points of $ C $ and the set of places $ V_K $. In particular, an open affine subset $ C_0 \subseteq C $ corresponds to a ring of integers $ \OO_K \subseteq K $ whose fraction field is $ K $, and its complement $ C \setminus C_0 $ corresponds to finitely many infinite places. For each place $ v \in V_K $, write
$$ \OO_v \coloneqq \OO_{K_v}, \ |\cdot|_v \coloneqq |\cdot|_{K_v}, \ q_v \coloneqq q_{K_v}, \ W_v \coloneqq W_{K_v}, \ I_v \coloneqq I_{K_v}, \ \phi_v \coloneqq \phi_{K_v}. $$
Furthermore, let $ G_K $ be the absolute Galois group of $ K $, and let $ \overline{G_K} \coloneqq G_{\overline{\F_q}(C)} $ be its quotient by the smallest closed subgroup containing $ I_v $ for all places $ v \in V_K $.

As in the local case, a \textbf{$ \lambda $-adic representation} of $ K $ is a continuous homomorphism $ \rho_\lambda : G_K \to \GL(V_\lambda) $, where $ \GL(V_\lambda) $ is endowed with the $ \lambda $-adic topology. Assume throughout that $ \rho_\lambda $ is \emph{almost everywhere unramified}, namely that its restrictions $ \rho_{\lambda, v} \coloneqq \rho_\lambda|_{W_v} : W_v \to \GL(V_\lambda) $ are unramified as $ \lambda $-adic representations of $ K_v $ at all places $ v $ in an open subset $ C_0 \subseteq C $. It turns out that there is an explicit equivalence between the categories of $ \lambda $-adic representations unramified on $ C_0 $ and of lisse $ \lambda $-adic sheaves on $ C_0 $ \cite[Proposition 1.2.5]{Jou77}. Let $ \FF_{\rho_\lambda} $ denote the push-forward of this lisse sheaf to $ C $, which is a constructible sheaf on $ C $.

A $ \lambda $-adic representation $ \rho_\lambda $ is said to be \textbf{self-dual of weight $ w(\rho_\lambda) $ and sign $ c(\rho_\lambda) $} in the sense of Ulmer \cite[Section 3.1.4]{Ulm05} if
\begin{itemize}
\item it is \textbf{$ \iota $-pure of weight $ w(\rho_\lambda) $} for some fixed embedding $ \iota : F_\lambda \hookrightarrow \C $, namely that there is some integer $ w(\rho_\lambda) \in \Z $ such that each eigenvalue of $ \rho_\lambda(\phi_v) $ has modulus $ q_v^{w(\rho_\lambda) / 2} $ under $ \iota $ at all unramified places $ v \in V_K $, and
\item there is some field endomorphism $ c(\rho_\lambda) : F_\lambda \to F_\lambda $ such that
$$ \rho_\lambda^\vee \cong \rho_\lambda^{c(\rho_\lambda)} \otimes F_\lambda(w(\rho_\lambda)), $$
where $ \rho_\lambda^\vee $ is the contragredient of $ \rho_\lambda $ and $ F_\lambda(w(\rho_\lambda)) $ is the Tate twist by $ w(\rho_\lambda) $, so that $ F_\lambda(w(\rho_\lambda))|_{W_v} = |\cdot|_v^{w(\rho_\lambda)} $ at all unramified places $ v \in V_K $.
\end{itemize}
By Lafforgue's theorem, if $ \rho_\lambda|_{\overline{G_K}} $ is absolutely irreducible, then it corresponds to a unique cuspidal automorphic representation of $ \GL_r(\A_K) $ in the sense of Langlands, and hence is $ \iota $-pure of some weight and independent of $ \lambda $ \cite[Theorem VII.6]{Laf02}.

In many cases, $ c(\rho_\lambda) $ is either the identity $ \id $ or complex conjugation $ \cc $.

\begin{example}
As in the local case, an abelian variety $ A $ over $ K $ has a natural $ \lambda $-adic representation $ \rho_{A, \ell} $, which is self-dual of weight $ 1 $ and sign $ \id $.
\end{example}

\begin{example}
As in the local case, an \emph{Artin representation} of $ K $ is a continuous homomorphism $ \varrho : G_K \to \GL(V) $, which is self-dual of weight $ 0 $ and some sign. When $ \varrho $ is one-dimensional, these correspond to Dirichlet characters with sign $ \cc $. While they are often of special interest, Artin representations with sign $ \id $ are expected to be rare in higher dimensions \cite{Roh13}.
\end{example}

The \textbf{Artin conductor} of $ \rho_\lambda $ is the effective multiplicative divisor
$$ \ff(\rho_\lambda) \coloneqq \prod_{v \in V_K} v^{\alpha(\rho_{\lambda, v})}, $$
which is a finite product since $ \rho_\lambda $ is almost everywhere unramified by assumption. When $ \rho_\lambda = \rho_{E, \ell} $ arises from an elliptic curve $ E $ over $ K $, this is the analogue for $ K $ of the usual conductor of $ E $ as seen in the LMFDB.

\pagebreak

For each place $ v \in V_K $, let $ \psi $ be a non-trivial additive character of $ \A_K / K $, and let $ \mu $ be an additive Haar measure on $ \A_K $ that is trivial on $ \A_K / K $ and on $ \OO_v $ for all but finitely many places $ v \in V_K $. The \textbf{global $ \epsilon $-factor} of $ \rho_\lambda $ is then
$$ \epsilon(\rho_\lambda) \coloneqq \prod_{v \in V_K} \epsilon(\rho_{\lambda, v}, \psi_v, \mu_v), $$
where $ \psi_v \coloneqq \psi|_{K_v} $ and $ \mu_v \coloneqq \mu|_{K_v} $. This is a finite product since $ \gamma(\psi_v) = 0 $ and $ \mu_v(\OO_v) = 1 $ for all but finitely many places $ v \in V_K $, and is independent of the choices of $ \{\psi_v\}_{v \in V_K} $ \cite[Equation 5.4]{Del73b} and of $ \{\mu_v\}_{v \in V_K} $ \cite[Equation 5.3]{Del73b}. Thus it suffices to choose $ \mu_v $ to be self-dual with respect to the Fourier inversion formula on $ \psi_v $ for each place $ v \in V_K $, so that $ \mu_v(\OO_v) = q_v^{-\gamma(\psi_v) / 2} $ \cite[Corollary 3 of Proposition VII.2.2]{Wei74}, which defines an additive Haar measure $ \mu \coloneqq \bigotimes_{v \in V_K} \mu_v $ on $ \A_K $ that is trivial on $ \A_K / K $ \cite[Corollary 2 of Theorem VII.2.1]{Wei74}.

\begin{proposition}
\label{prop:modulusepsilon}
Let $ \rho_\lambda $ be a $ \lambda $-adic representation of a global function field $ K $ that is self-dual of weight $ w(\rho_\lambda) $ and sign $ c(\rho_\lambda) $. Then
$$ |\epsilon(\rho_\lambda)| = q^{(\deg\ff(\rho_\lambda) + (2g(C) - 2)\dim\rho_\lambda)(w(\rho_\lambda) + 1) / 2}. $$
\end{proposition}

\begin{proof}
For each place $ v \in V_K $, \cite[Proposition (i) of Section 12]{Roh94}
$$ |\epsilon(\rho_\lambda, \psi_v, \mu_v)| = q_v^{(\gamma(\psi_v)\dim\rho_\lambda + \alpha(\rho_{\lambda, v}))(w(\rho_\lambda) + 1) / 2}, $$
so the result follows from $ \prod_{v \in V_K} q_v^{\gamma(\psi_v)} = q^{2g(C) - 2} $ \cite[Proposition VII.2.6]{Wei74}.
\end{proof}

\subsection{L-functions}

For each place $ v \in V_K $, the \textbf{formal local Euler factor} of $ \rho_\lambda $ at $ v $ is the inverse characteristic polynomial of $ \rho_{\lambda, v}^{I_v}(\phi_v) \in \GL(V_\lambda) $, namely
$$ \LL_v(\rho_\lambda, T) \coloneqq \det(1 - \rho_{\lambda, v}^{I_v}(\phi_v) \cdot T^{\deg v}) \in 1 + T \cdot F_\lambda[T]. $$
The \textbf{formal L-function} of $ \rho_\lambda $ is then the formal power series
$$ \LL(\rho_\lambda, T) \coloneqq \prod_{v \in V_K} \LL_v(\rho_\lambda, T)^{-1} \in 1 + T \cdot F_\lambda[[T]], $$
which recovers the usual L-function of $ \rho_\lambda $ by $ L(\rho_\lambda, s) \coloneqq \LL(\rho_\lambda, q^{-s}) $. As a consequence of the proof of the Weil conjectures, $ \LL(\rho_\lambda, T) $ is known to be a rational function satisfying a functional equation and the Riemann hypothesis.

\begin{theorem}
\label{thm:weilconjectures}
Let $ \rho_\lambda $ be an almost everywhere unramified $ \lambda $-adic representation of a global function field $ K = \F_q(C) $.
\begin{enumerate}
\item Rationality: there are unique coprime polynomials $ N(\rho_\lambda, T), D(\rho_\lambda, T) \in 1 + T \cdot F_\lambda[T] $ of degrees $ n(\rho_\lambda), d(\rho_\lambda) \in \N $ respectively such that
$$ \LL(\rho_\lambda, T) = \dfrac{N(\rho_\lambda, T)}{D(\rho_\lambda, T)}. $$
\item Integrality: if $ \rho_\lambda $ has no $ \overline{G_K} $-invariants, then $ d(\rho_\lambda) = 0 $.
\item Degree formula:
$$ n(\rho_\lambda) - d(\rho_\lambda) = \deg\ff(\rho_\lambda) + (2g(C) - 2)\dim\rho_\lambda. $$
\item Functional equation:
$$ \LL(\rho_\lambda, T) = \epsilon(\rho_\lambda) \cdot T^{n(\rho_\lambda) - d(\rho_\lambda)} \cdot \LL(\rho_\lambda^\vee, (qT)^{-1}). $$
\item Riemann hypothesis: if $ \rho_\lambda $ is $ \iota $-pure of weight $ w(\rho_\lambda) $, then the roots of $ N(\rho_\lambda, T) $ have modulus $ q^{-(w(\rho_\lambda) + 1) / 2} $.
\end{enumerate}
\end{theorem}

\pagebreak

\begin{proof}
Let $ H_\et^r(\overline{C}, \FF_{\rho_\lambda}) $ be the $ r $-th $ \lambda $-adic cohomology groups of $ \FF_{\rho_\lambda} $ for each $ r \in \{0, 1, 2\} $, and let $ P_r(\rho_\lambda, T) \coloneqq \det(1 - H_\et^r(\overline{C}, \FF_{\rho_\lambda})(\varphi_q) \cdot T) $ be the characteristic polynomials of $ q $-Frobenius $ \varphi_q \in \Gal(\overline{\F_q} / \F_q) $ acting on $ H_\et^r(\overline{C}, \FF_{\rho_\lambda}) $. Then set
$$ N(\rho_\lambda, T) \coloneqq P_1(\rho_\lambda, T), \qquad D(\rho_\lambda, T) \coloneqq P_0(\rho_\lambda, T) \cdot P_2(\rho_\lambda, T). $$
The first statement follows from the Grothendieck--Lefschetz trace formula \cite[Theorem 5.1]{Gro68}. The second statement follows from Poincar\'e duality \cite[Equation 3.2.6.2]{Del73a}. The third statement is the Grothendieck--Ogg--Shafarevich formula \cite[Theorem 7.1]{Gro77}. The fourth statement follows from Deligne's functional equation \cite[Theorem 9.3]{Del73b}. The fifth statement follows from Deligne's purity theorem \cite[Theorem 3.3.1]{Del80}. The main arguments with the required level of generality are also sketched by Ulmer \cite[Theorem 1.3.3 and Theorem 1.4.1]{Ulm11}.
\end{proof}

Theorem \ref{thm:weilconjectures} is central in the computation of $ \LL(\rho_\lambda, T) $, which will be explained now. Part (1) reduces its computation to that of $ N(\rho_\lambda, T) $ and $ D(\rho_\lambda, T) $. Now $ D(\rho_\lambda, T) $ is typically known, and in fact often trivial by part (2), so it remains to compute $ N(\rho_\lambda, T) $. Part (3) reduces the computation of its degree $ n(\rho_\lambda) $ to that of $ \ff(\rho_\lambda) $, which is simply a local calculation. This is sufficient to determine the coefficients of $ N(\rho_\lambda, T) $ by computing $ \LL_v(\rho_\lambda, T) $ for sufficiently many places $ v \in V_K $. Parts (4) and (5) are auxiliary in the computation of $ \LL(\rho_\lambda, T) $. Assuming that $ \rho_\lambda $ is self-dual of weight $ w(\rho_\lambda) $ and sign $ c(\rho_\lambda) \in \{\id, \cc\} $, then
$$ \LL(\rho_\lambda^\vee, (qT)^{-1}) = \LL(\rho_\lambda^{c(\rho_\lambda)} \otimes F_\lambda(w(\rho_\lambda)), (qT)^{-1}) = \LL(\rho_\lambda, (q^{w(\rho_\lambda) + 1}T)^{-1})^{c(\rho_\lambda)}, $$
so part (4) says that
$$ \LL(\rho_\lambda, T) = \epsilon(\rho_\lambda) \cdot T^{n(\rho_\lambda) - d(\rho_\lambda)} \cdot \LL(\rho_\lambda, (q^{w(\rho_\lambda) + 1}T)^{-1})^{c(\rho_\lambda)}, $$
which gives a relation on the coefficients of $ N(\rho_\lambda, T) $. In particular, if $ \epsilon(\rho_\lambda) $ is known, then this significantly reduces the number of places $ v \in V_K $ necessary in the previous computation. Otherwise, it enforces relations between $ \epsilon(\rho_\lambda) $ and the coefficients of $ N(\rho_\lambda, T) $ that allows for an independent computation of $ \epsilon(\rho_\lambda) $. Both parts (4) and (5) also serve as sanity checks to test for implementation bugs.

\begin{remark}
The condition that $ \rho_\lambda $ has no $ \overline{G_K} $-invariants, and hence Theorem \ref{thm:weilconjectures}(2), applies in many situations. Notably, if $ \rho_\lambda = \rho_{E, \ell} $ arises from a non-constant elliptic curve $ E $ over $ K $, then $ E(\overline{\F_q}(C)) $ is finitely generated by the Lang--N\'eron theorem, so $ E(\overline{\F_q}(C))[\ell^r] $ eventually stabilises for increasing $ r \in \N $, and hence $ \rho_{E, \ell} $ has no $ \overline{G_K} $-invariants. Note that this argument fails for an abelian variety $ A $ over $ K $, since its $ (\overline{\F_q}(C) / \overline{\F_q}) $-trace is not trivial in general \cite[Section 2]{Con06}.

Since $ I_v \subseteq \overline{G_K} $ for any place $ v \in V_K $, this also applies whenever $ \rho_\lambda $ has no $ I_v $-invariants for some place $ v \in V_K $. For instance, this applies when $ A $ has purely additive reduction at some place $ v \in V_K $, where the $ I_v $-invariant subgroup of $ \rho_\lambda = \rho_{A, \ell} $ is the identity component of the special fibre of its N\'eron model at $ v $ \cite[Lemma 2]{ST68}, which is trivial since the Tate module of the additive group is trivial.

More generally, a tensor product of $ \lambda $-adic representations, one of which with no $ I_v $-invariants at some place $ v \in V_K $, has no $ \overline{G_K} $-invariants if their Artin conductors are coprime. This is the case for twists of $ A $ by irreducible Artin representations $ \varrho $ such that $ (\ff(A), \ff(\varrho)) = 1 $ and $ \ff(\varrho) \ne 1 $. The latter condition applies to twists of $ A $ by primitive Dirichlet characters that factor through separable geometric extensions of the genus zero global function field $ \F_q(t) $, which cannot be everywhere unramified by the Riemann--Hurwitz formula \cite[Theorem 7.16]{Ros02}.
\end{remark}

\pagebreak

\section{Computing L-functions}
\label{sec:computingfunctions}

With the application of Theorem \ref{thm:weilconjectures} in mind, the setup in this section can be made completely abstract, albeit with suggestive notations. In what follows, let $ F $ be a field, and let $ V $ be a possibly infinite indexing set.

\begin{notation}
For any polynomial or power series $ P(T) \in F[[T]] $ and any $ i \in \N $, let $ P_i $ be the $ i $-th coefficient of $ P(T) $, and let $ [P(T)]_i \coloneqq \sum_{k = 0}^i P_kT^k \in 1 + T \cdot F[[T]] $.
\end{notation}

\begin{definition}
A \textbf{stratification} is a function $ \LL_\bullet(T) : V \to 1 + T \cdot F[T] $ such that
$$ V_i \coloneqq \{v \in V : \LL_v(T)_i \ne 0, \ [\LL_v(T)]_{i - 1} = 1\} $$
is finite for each $ i \in \N $. In this case, define the formal product
$$ \LL(T) \coloneqq \prod_{v \in V} \LL_v(T)^{-1} \in 1 + T \cdot F[[T]], $$
which is well-defined as each coefficient of $ \LL(T) $ is a finite combination of coefficients of $ \LL_v(T) $ for finitely many indices $ v \in V $. Note that the data of $ F $ and $ V $ are implicit.
\end{definition}

In what follows, functions with values in $ F $ or $ F[T] $ should be \textbf{computable}, in the na\"ive sense that they can be evaluated within reasonable time. For a stratification $ \LL_\bullet(T) : V \to 1 + T \cdot F[T] $, the running time of its evaluation at any $ v \in V $ should preferably be polynomial in the number of bits representing $ v $.

\begin{example}
If $ \rho_\lambda $ is an almost everywhere unramified $ \lambda $-adic representation of a global function field $ K $, then its formal local Euler factor $ \LL_\bullet(T) = \LL_\bullet(\rho_\lambda, T) $ is a stratification of $ V_K $, in which case $ \LL(T) = \LL(\rho_\lambda, T) $ is its formal L-function.
\end{example}

\subsection{Rationality}

Assume that $ \LL(T) $ is the quotient of some polynomial $ N(T) \in 1 + T \cdot F[T] $ of known degree $ n \in \N $ by some known polynomial $ D(T) \in 1 + T \cdot F[T] $ of degree $ d \in \N $. Then the coefficients of $ N(T) = \LL(T) \cdot D(T) $ can be determined completely by taking the product of $ \LL_v(T)^{-1} $ for sufficiently many indices $ v \in V $.

\begin{algorithm}
\label{alg:rationalfunction}
Input: $ (\LL_\bullet(T), D(T), n_0, n_1, n') $, where $ \LL_\bullet(T) : V \to 1 + T \cdot F[T] $ is a computable stratification, $ D(T) \in 1 + T \cdot F[T] $ is a polynomial, and $ n_0, n_1, n' \in \N $ are integers with $ n_0 \le n_1 \le n' $.
\begin{enumerate}
\item Initialise $ P(T) \coloneqq 1 $.
\item For $ i \in \{n_0, \dots, n_1\} $ and for $ v \in V_i $, replace $ P(T) \coloneqq [P(T) \cdot [\LL_v(T)]_{n'}]_{n'} $.
\item Replace $ P(T) \coloneqq [P(T)^{-1}]_{n'} $.
\item Return $ P(T) \coloneqq [P(T) \cdot [D(T)]_{n'}]_{n'} $.
\end{enumerate}
Output: a polynomial $ P(T) \in 1 + T \cdot F[T] $.
\end{algorithm}

\begin{proposition}
\label{prop:rationalfunction}
Let $ \LL_\bullet(T) : V \to 1 + T \cdot F[T] $ be a stratification with
$$ \LL(T) = \dfrac{N(T)}{D(T)}, $$
for some polynomials $ N(T), D(T) \in 1 + T \cdot F[T] $ of degrees $ n, d \in \N $ respectively, and let $ n_0, n_1, n' \in \N $ be such that $ \{0, 1\} \ni n_0 \le n_1 \le n' \le n $.
\begin{enumerate}
\item Algorithm \ref{alg:rationalfunction} with inputs $ (\LL_\bullet(T), D(T), n_0, n_1, n') $ outputs a polynomial $ P(T) \in 1 + T \cdot F[T] $ of degree $ n' $ such that $ [P(T)]_{n_1} = [N(T)]_{n_1} $.
\item If $ n_1' \in \N $ is such that $ n_1 < n_1' \le n' $, then Algorithm \ref{alg:rationalfunction} with inputs $ (\LL_\bullet(T), 1, n_1 + 1, n_1', n') $ outputs a polynomial $ Q(T) \in 1 + T \cdot F[T] $ of degree $ n' $ such that $ [P(T) \cdot Q(T)]_{n_1'} = [N(T)]_{n_1'} $.
\end{enumerate}
\end{proposition}

\pagebreak

\begin{proof}
Since $ V_0 = \emptyset $ and $ \LL_v(T)^{-1} \equiv \LL_v(T) \equiv 1 \mod T^{n_1 + 1} $ for any $ v \in V_i $ whenever $ i > n_1 $, Algorithm \ref{alg:rationalfunction} with inputs $ (\LL_\bullet(T), D(T), n_0, n_1, n') $ computes
$$ P(T) = [D(T)]_{n'} \cdot \prod_{i = n_0}^{n_1} \prod_{v \in V_i} [\LL_v(T)^{-1}]_{n'} \equiv D(T) \cdot \LL(T) \mod T^{n_1 + 1}, $$
so the first statement follows immediately. Multiplying $ P(T) $ with the output of Algorithm \ref{alg:rationalfunction} with inputs $ (\LL_\bullet(T), 1, n_1 + 1, n_1', n') $ increases the range of the product to $ i \in \{n_0, \dots, n_1'\} $, so the second statement also follows.
\end{proof}

In particular, Algorithm \ref{alg:rationalfunction} with inputs $ (\LL_\bullet(T), D(T), 1, n, n) $ outputs $ N(T) $, but the extra degrees of freedom on the last three integers will be useful later.

\subsection{Functional equation}

Now in addition to rationality, assume further that $ \LL(T) $ satisfies a functional equation of the form $ \LL(T) = \epsilon \cdot T^{n - d} \cdot \LL(f / T)^c $, for some known field elements $ \epsilon, f \in F $ and some known field endomorphism $ c : F \to F $.

\begin{algorithm}
\label{alg:computecoefficients}
Input: $ (\LL_\bullet(T), D(T), m, n_1, n', f, c) $, where $ \LL_\bullet(T) : V \to 1 + T \cdot F[T] $ is a computable stratification, $ D(T) \in 1 + T \cdot F[T] $ is a polynomial, $ m, n_1, n' \in \N $ are integers with $ m \le n_1 \le n' $, $ f \in F $ is a field element, and $ c : F \to F $ is a computable field endomorphism.
\begin{enumerate}
\item Run Algorithm \ref{alg:rationalfunction} with inputs $ (\LL_\bullet(T), D(T), 1, n_1, n') $ and output $ P(T) $.
\item Set $ d \coloneqq \deg D(T) $.
\item For $ k \in \{0, \dots, m\} $:
\begin{enumerate}
\item Set $ h \coloneqq 1 $ and $ k' \coloneqq \min(d, k) $, and initialise $ M_k \coloneqq 0 $.
\item For $ i \in \{0, \dots, k\} $:
\begin{enumerate}
\item If $ 0 \le k - i \le k' $, then replace $ M_k \coloneqq M_k + P_i^c \cdot D_{d - (k - i)} \cdot h $.
\item If $ 1 \le i \le k' $, then replace $ M_k \coloneqq M_k - M_{k - i} \cdot D_i^c \cdot h $.
\item Replace $ h \coloneqq h \cdot f $.
\end{enumerate}
\end{enumerate}
\item Return $ (P(T), M_0, \dots, M_m) $.
\end{enumerate}
Output: a polynomial $ P(T) \in 1 + T \cdot F[T] $ and field elements $ M_0, \dots, M_m \in F $.
\end{algorithm}

\begin{remark}
The number of additions in Algorithm \ref{alg:computecoefficients} can be optimised when $ k \le d $, in which case the ranges of both conditionals overlap for all $ i \in \{1, \dots, k\} $. In any case, it might be interesting to obtain a closed form expression for all $ M_k $, in contrast to a recursive procedure. Note that when $ D(T) = 1 $, which is often the case, Algorithm \ref{alg:computecoefficients} computes the simple relation $ M_k = P_k^cf^k $ for all $ k \in \{0, \dots, m\} $.
\end{remark}

\begin{proposition}
\label{prop:computecoefficients}
Let $ \LL_\bullet(T) : V \to 1 + T \cdot F[T] $ be a stratification with
$$ \LL(T) = \dfrac{N(T)}{D(T)}, $$
for some polynomials $ N(T), D(T) \in 1 + T \cdot F[T] $ of degrees $ n, d \in \N $ respectively, satisfying a functional equation
$$ \LL(T) = \epsilon \cdot T^{n - d} \cdot \LL(f / T)^c, $$
for some field elements $ \epsilon, f \in F $ and some field endomorphism $ c : F \to F $, and let $ m, n' \in \N $ be such that $ m \le \lfloor n / 2\rfloor \le n' \le n $. Then Algorithm \ref{alg:computecoefficients} with inputs $ (\LL_\bullet(T), D(T), m, \lfloor n / 2\rfloor, n', f, c) $ outputs a polynomial $ P(T) \in 1 + T \cdot F[T] $ of degree $ n' $ such that $ [P(T)]_{\lfloor n / 2\rfloor} = [N(T)]_{\lfloor n / 2\rfloor} $ and field elements $ M_k \in F $ such that $ N_{n - k} = \epsilon M_k $ for all $ k \in \{0, \dots, m\} $.
\end{proposition}

\pagebreak

\begin{proof}
The functional equation says for each $ r \in \{-d, \dots, n\} $,
$$ \underset{i - j = r}{\sum_{i = 0}^n \sum_{j = 0}^d} N_iD_j^cf^j = \underset{(n - i) - (d - j) = r}{\sum_{i = 0}^n \sum_{j = 0}^d} \epsilon N_i^cD_jf^i. $$
By expressing $ j \coloneqq i - r $ in the left hand side, with a further replacement $ i - r \mapsto i $, and by expressing $ j \coloneqq d - ((n - r) - i) $ in the right hand side,
$$ \sum_{i = -\min(0, r)}^{\min(d, n - r)} N_{r + i}D_i^cf^i = \sum_{(n - r) - i = -\min(0, r)}^{\min(d, n - r)} \epsilon N_i^cD_{d - ((n - r) - i)}f^i. $$
In particular, setting $ r \coloneqq n - k $ for each $ k \in \{0, \dots, n\} $, a rearrangement gives
$$ N_{n - k} = \sum_{k - i = 0}^{\min(d, k)} \epsilon N_i^cD_{d - (k - i)}f^i - \sum_{i = 1}^{\min(d, k)} N_{n - (k - i)}D_i^cf^i. $$
By Proposition \ref{prop:rationalfunction}(1), step (1) of Algorithm \ref{alg:computecoefficients} computes a polynomial $ P(T) \in 1 + T \cdot F[T] $ of degree $ n $ such that $ P_i = N_i $ for all $ i \in \{0, \dots, \lfloor n / 2\rfloor\} $, and in particular for all $ i \in \{0, \dots, k\} $ and all $ k \in \{0, \dots, m\} $. The remaining steps then compute
$$ M_k \coloneqq \sum_{k - i = 0}^{\min(d, k)} P_i^cD_{d - (k - i)}f^i - \sum_{i = 1}^{\min(d, k)} M_{k - i}D_i^cf^i, $$
so the result follows by induction.
\end{proof}

Algorithm \ref{alg:computecoefficients} removes the need to compute $ V_{\lfloor n / 2\rfloor + 1}, \dots, V_n $.

\begin{algorithm}
\label{alg:functionalequation}
Input: $ (\LL_\bullet(T), D(T), n, \epsilon, f, c) $, where $ \LL_\bullet(T) : V \to 1 + T \cdot F[T] $ is a computable stratification, $ D(T) \in 1 + T \cdot F[T] $ is a polynomial, $ n \in \N $ is an integer, $ e, f \in F $ are field elements, and $ c : F \to F $ is a computable field endomorphism.
\begin{enumerate}
\item Set $ n_1 = \lfloor n / 2\rfloor $ and $ m \coloneqq n - n_1 - 1 $.
\item Run Algorithm \ref{alg:computecoefficients} with inputs $ (\LL_\bullet(T), D(T), m, n_1, n_1, f, c) $ and output $ P(T) $ and $ (M_0, \dots, M_m) $.
\item For $ k \in \{n - m, \dots, n\} $, replace $ P(T) \coloneqq P(T) + \epsilon M_{n - k}T^k $.
\item Return $ P(T) $.
\end{enumerate}
Output: a polynomial $ P(T) \in 1 + T \cdot F[T] $.
\end{algorithm}

\begin{proposition}
\label{prop:functionalequation}
Let $ \LL_\bullet(T) : V \to 1 + T \cdot F[T] $ be a stratification with
$$ \LL(T) = \dfrac{N(T)}{D(T)}, $$
for some polynomials $ N(T), D(T) \in 1 + T \cdot F[T] $ of degrees $ n, d \in \N $ respectively, satisfying a functional equation
$$ \LL(T) = \epsilon \cdot T^{n - d} \cdot \LL(f / T)^c, $$
for some field elements $ \epsilon, f \in F $ and some field endomorphism $ c : F \to F $. Then Algorithm \ref{alg:functionalequation} with inputs $ (\LL_\bullet(T), D(T), n, \epsilon, f, c) $ outputs $ N(T) $.
\end{proposition}

\begin{proof}
By Proposition \ref{prop:computecoefficients}, step (2) of Algorithm \ref{alg:functionalequation} computes the polynomial $ P(T) = \sum_{i = 0}^{\lfloor n / 2\rfloor} N_iT^i $ and field elements $ M_k \in F $ such that $ N_k = \epsilon M_{n - k} $ for all $ k \in \{\lfloor n / 2\rfloor + 1, \dots, n\} $. The remaining steps then compute $ N(T) = P(T) + \sum_{k = \lfloor n / 2\rfloor + 1}^n \epsilon M_{n - k}T^k $, so the result follows.
\end{proof}

\pagebreak

\subsection{\texorpdfstring{$ \epsilon $-}{epsilon }factors}

In some cases, the field element $ \epsilon $ cannot be provided as an input. Algorithm \ref{alg:computecoefficients} gives an independent way to compute it.

\begin{algorithm}
\label{alg:computeepsilon}
Input: $ (\LL_\bullet(T), D(T), n, f, c) $, where $ \LL_\bullet(T) : V \to 1 + T \cdot F[T] $ is a computable stratification, $ D(T) \in 1 + T \cdot F[T] $ is a polynomial, $ n \in \N $ is an integer, $ f \in F $ is a field element, and $ c : F \to F $ is a computable field endomorphism.
\begin{enumerate}
\item Set $ m \coloneqq n_1 \coloneqq \lceil n / 2\rceil $.
\item Run Algorithm \ref{alg:computecoefficients} with inputs $ (\LL_\bullet(T), D(T), m, n_1, n, f, c) $ and output $ P(T) $ and $ (M_0, \dots, M_m) $.
\item Choose $ k \in \{n - m, \dots, n\} $ minimal such that $ M_{n - k} \ne 0 $.
\item If $ n_1 < k $, then run Algorithm \ref{alg:rationalfunction} with inputs $ (\LL_\bullet(T), 1, n_1 + 1, k, k) $ and output $ Q(T) $.
\item Set $ R(T) \coloneqq [P(T) \cdot Q(T)]_k $.
\item Return $ \epsilon \coloneqq R_k / M_{n - k} $.
\end{enumerate}
Output: a field element $ \epsilon \in F $.
\end{algorithm}

\begin{proposition}
\label{prop:computeepsilon}
Let $ \LL_\bullet(T) : V \to 1 + T \cdot F[T] $ be a stratification with
$$ \LL(T) = \dfrac{N(T)}{D(T)}, $$
for some polynomials $ N(T), D(T) \in 1 + T \cdot F[T] $ of degrees $ n, d \in \N $ respectively, satisfying a functional equation
$$ \LL(T) = \epsilon \cdot T^{n - d} \cdot \LL(f / T)^c, $$
for some field elements $ \epsilon, f \in F $ and some field endomorphism $ c : F \to F $. Then Algorithm \ref{alg:functionalequation} with inputs $ (\LL_\bullet(T), D(T), n, f, c) $ outputs $ \epsilon $.
\end{proposition}

\begin{proof}
By Proposition \ref{prop:computecoefficients}, step (2) of Algorithm \ref{alg:computeepsilon} computes a polynomial $ P(T) \in 1 + T \cdot F[T] $ of degree $ n $ such that $ P_i = N_i $ for all $ i \in \{0, \dots, \lceil n / 2\rceil\} $ and field elements $ M_k \in F $ such that $ N_k = \epsilon M_{n - k} $ for all $ k \in \{\lfloor n / 2\rfloor, \dots, n\} $, and step (3) terminates for some $ k \in \{n - m, \dots, n\} $ since $ M_0 = D_d \ne 0 $. By Proposition \ref{prop:rationalfunction}(2), the remaining steps compute a polynomial $ R(T) \in 1 + T \cdot F[T] $ of degree $ k $ such that $ R_k = N_k = \epsilon M_{n - k} $, so the result follows.
\end{proof}

\subsection{\texorpdfstring{$ \lambda $}{l}-adic representations}

Let $ \rho_\lambda $ be an almost everywhere unramified $ \lambda $-adic representation of $ K = \F_q(C) $, where $ \lambda \nmid q $ is a fixed prime of a coefficient number field $ F $. Then $ \LL(\rho_\lambda, T) $ is a rational function in $ F_\lambda(T) $ by Theorem \ref{thm:weilconjectures}(1), so the following is immediate from Proposition \ref{prop:rationalfunction}(1).

\begin{corollary}
\label{cor:rationalfunction}
Let $ \rho_\lambda $ be an almost everywhere unramified $ \lambda $-adic representation of a global function field $ K $. Then Algorithm \ref{alg:rationalfunction} outputs $ N(\rho_\lambda, T) \in 1 + T \cdot F_\lambda[T] $ with inputs
$$ (\LL_\bullet(\rho_\lambda, T), \ D(\rho_\lambda, T), \ 1, \ n(\rho_\lambda), \ n(\rho_\lambda)). $$
\end{corollary}

For many explicit $ \lambda $-adic representations $ \rho_\lambda $ of a global function field $ K = \F_q(C) $, the polynomial $ D(\rho_\lambda, T) $ and the integer $ n(\rho_\lambda) $ are easily computable, and $ \LL_v(\rho_\lambda, T) $ should be computable with a running time of $ O(b_{\rho_\lambda}(\deg v)) $ for some function $ b_{\rho_\lambda} : \N \to \N $ that depends only on $ \deg v $ rather than $ v $ itself. For instance, when $ \rho_\lambda = \rho_{E, \ell} $ arises from an elliptic curve $ E $ over $ K $, computing the only non-trivial coefficient of $ \LL_v(\rho_{E, \ell}, T) $ reduces to computing $ \#E(\F_{q^{\deg v}}) $ via Schoof's algorithm, which has a running time of $ \widetilde{O}((\deg v \cdot \log q)^4) $ in the number of operations on $ F = \Q $ \cite{Sch85}, while $ D(\rho_\lambda, T) $ and $ n(\rho_\lambda) $ can be determined by Tate's algorithm, which is very efficient.

\pagebreak

Let $ n_1 = n' = n(\rho_\lambda) $. The running time of Algorithm \ref{alg:rationalfunction} with the inputs of Corollary \ref{cor:rationalfunction} is clearly dominated by step (2), which involves computing $ \LL_v(T) $ and a multiplication modulo $ T^{n'} $. The latter can be done via the Sch\"onhage--Strassen multiplication algorithm with a running time of $ O(n' \cdot \log n' \cdot \log\log n') = \widetilde{O}(n') $ in the number of operations on $ F = F_\lambda $ \cite{SS71}. By the prime number theorem for global function fields \cite[Theorem 5.12]{Ros02}, there are $ O(q^i / i) $ places $ v \in V_K $ of degree $ \deg v = i $ for each $ i \in \{1, \dots, n_1\} $, so the total running time would be
$$ O\left(\sum_{i = 1}^{n_1} \dfrac{q^i}{i} \cdot (b_{\rho_\lambda}(i) + n' \cdot \log n' \cdot \log\log n')\right) = \widetilde{O}(q^{n_1} \cdot (b_{\rho_\lambda}(n_1) + n')). $$

\begin{remark}
While $ b_{\rho_\lambda}(n_1) $ usually dominates $ n' $, this is not always the case, especially if $ \LL_v(T) $ can be obtained from a lookup table. In practice, both are dominated by $ q^{n_1} $, so the total running time is roughly exponential in $ n_1 $. In other words, the running time is typically dominated by the enumeration of $ V_i $ for all $ i \in \{1, \dots, n_1\} $ rather than the actual multiplications of $ \LL_v(T) $ modulo $ T^{n'} $.
\end{remark}

Assuming that $ \rho_\lambda $ is self-dual of weight $ w(\rho_\lambda) $ and sign $ c(\rho_\lambda) \in \{\id, \cc\} $, the following is immediate from Theorem \ref{thm:weilconjectures}(4) and Proposition \ref{prop:functionalequation}.

\begin{corollary}
\label{cor:functionalequation}
Let $ \rho_\lambda $ be an almost everywhere unramified $ \lambda $-adic representation of a global function field $ K = \F_q(C) $ that is self-dual of weight $ w(\rho_\lambda) $ and sign $ c(\rho_\lambda) $. Then Algorithm \ref{alg:functionalequation} outputs $ N(\rho_\lambda, T) \in 1 + T \cdot F_\lambda[T] $ with inputs
$$ (\LL_\bullet(\rho_\lambda, T), \ D(\rho_\lambda, T), \ n(\rho_\lambda), \ \epsilon(\rho_\lambda), \ 1 / q^{w(\rho_\lambda) + 1}, \ c(\rho_\lambda)). $$
\end{corollary}

Let $ n = n(\rho_\lambda) $ and $ n_1 = \lfloor n / 2\rfloor $. Step (2) of Algorithm \ref{alg:functionalequation} with the inputs of Corollary \ref{cor:functionalequation} reduces to Algorithm \ref{alg:computecoefficients} with inputs $ (\LL_\bullet(\rho_\lambda, T), D(\rho_\lambda, T), n - n_1 - 1, n_1, n_1, 1 / q^{w(\rho_\lambda) + 1}, c(\rho_\lambda)) $. The running time of Algorithm \ref{alg:computecoefficients} is in turn dominated by its step (1), which has a running time of $ \widetilde{O}(q^{n_1} \cdot (b_{\rho_\lambda}(n_1) + n_1)) $ as in the previous discussion, since the remaining steps have a total running time of $ O((n - n_1 - 1)^2) = O(n_1^2) $. Thus Algorithm \ref{alg:functionalequation} also has a total running time of $ \widetilde{O}(q^{n_1} \cdot (b_{\rho_\lambda}(n_1) + n_1)) $, since it is clearly dominated by its step (2).

As with $ D(\rho_\lambda) $ and $ n(\rho_\lambda) $, both $ w(\rho_\lambda) $ and $ c(\rho_\lambda) $ are typically known, but $ \epsilon(\rho_\lambda) $ may be difficult to compute given its elusive non-constructive definition. The following is immediate from Theorem \ref{thm:weilconjectures}(4) and Proposition \ref{prop:computeepsilon}.

\begin{corollary}
\label{cor:computeepsilon}
Let $ \rho_\lambda $ be an almost everywhere unramified $ \lambda $-adic representation of a global function field $ K = \F_q(C) $ that is self-dual of weight $ w(\rho_\lambda) $ and sign $ c(\rho_\lambda) $. Then Algorithm \ref{alg:computeepsilon} outputs $ \epsilon(\rho_\lambda) \in F_\lambda^\times $ with inputs
$$ (\LL_\bullet(\rho_\lambda, T), \ D(\rho_\lambda, T), \ n(\rho_\lambda), \ 1 / q^{w(\rho_\lambda) + 1}, \ c(\rho_\lambda)). $$
\end{corollary}

Let $ n = n(\rho_\lambda) $ and $ n_1 \coloneqq \lceil n / 2\rceil $. Steps (1) and (2) of Algorithm \ref{alg:computeepsilon} subsume that of Algorithm \ref{alg:functionalequation}, and have the dominant running time of $ \widetilde{O}(q^{n_1} \cdot (b_{\rho_\lambda}(n_1) + n)) $ as in the previous discussion, so this computation should not be repeated. In the best case scenario where $ M_m \ne 0 $, step (3) sets $ k = n - m $ and step (4) becomes vacuous, so its total running time is $ \widetilde{O}(q^{n_1} \cdot (b_{\rho_\lambda}(n_1) + n)) $. In the worst case scenario where $ M_1 = \dots = M_m = 0 $, step (3) sets $ k = n $ and step (4) has a running time of $ \widetilde{O}(q^n \cdot (b_{\rho_\lambda}(n) + n)) $, which is equivalent to that of Algorithm \ref{alg:rationalfunction}.

\begin{remark}
Heuristically, this worst case scenario occurs very rarely, while the best case scenario almost always occurs. This was also observed in a footnote of Comeau-Lapointe--David--Lal\'in--Li for $ L(E, \chi, s) $ \cite[Footnote 5]{CLDLL22}.
\end{remark}

\pagebreak

\section{Example computations}
\label{sec:examplecomputations}

Here are some examples of $ \lambda $-adic representations $ \rho_\lambda $ of $ K = \F_q(C) $ showcasing the main algorithms, which are mostly worked out by hand except for the enumeration of $ V_K $ via the \texttt{Places} function in Magma. For brevity, let $ V_{K, \le r} $ denote the finite subset of $ V_K $ consisting of places of degree at most $ r \in \N $.

\subsection{Trivial representations}

Let $ \rho_\lambda = \1 $ be the trivial representation of $ K $. Then $ \LL(\1, T) $ is just the Dedekind $ \zeta $-function $ \zeta_K(T) $, with $ \LL_v(\1, T) = 1 - T^{\deg v} $ for any place $ v \in V_K $, and the invariants are
$$ D(\1, T) = (1 - T)(1 - qT), \ n(\1) = 2g(C), \ \epsilon(\1) = q^{g(C) - 1}, \ w(\1) = 0, \ c(\1) = \id. $$

\begin{example}
Let $ K $ be the function field of the hyperelliptic curve $ C $ of genus $ g(C) = 3 $ over $ \F_3 $ given by $ u^2 = t^7 - t + 1 $, so $ n(\1) = 6 $. By Corollary \ref{cor:functionalequation}, it suffices to compute $ \LL_v(\1, T) $ for all places $ v \in V_{K, \le n_1} = V_{K, \le 3} $, tabulated as follows.
$$
\begin{array}{|c|c|}
\hline
v & \LL_v(\1, T) \\
\hline
(1 / t, u / t^4) & \multirow{7}{*}{$ 1 - T $} \\
\cline{1-1}
(t - 1, u + t) & \\
\cline{1-1}
(t - 1, u + t + 1) & \\
\cline{1-1}
(t, u - 1) & \\
\cline{1-1}
(t, u + 1) & \\
\cline{1-1}
(t + 1, u + t) & \\
\cline{1-1}
(t + 1, u + t - 1) & \\
\hline
(t^2 - t - 1, u) & \multirow{3}{*}{$ 1 - T^2 $} \\
\cline{1-1}
(t^2 + t - 1, u - t + 1) & \\
\cline{1-1}
(t^2 + t - 1, u + t - 1) & \\
\hline
\end{array}
\qquad
\begin{array}{|c|c|}
\hline
v & \LL_v(\1, T) \\
\hline
(t^3 - t^2 + 1, u - t^2) & \multirow{10}{*}{$ 1 - T^3 $} \\
\cline{1-1}
(t^3 - t^2 + 1, u + t^2) & \\
\cline{1-1}
(t^3 - t^2 + t + 1, u - t^2 - t - 1) & \\
\cline{1-1}
(t^3 - t^2 + t + 1, u + t^2 + t + 1) & \\
\cline{1-1}
(t^3 - t + 1, u - t^2 + t - 1) & \\
\cline{1-1}
(t^3 - t + 1, u + t^2 - t + 1) & \\
\cline{1-1}
(t^3 + t^2 - t + 1, u - t^2 - t + 1) & \\
\cline{1-1}
(t^3 + t^2 - t + 1, u + t^2 + t - 1) & \\
\cline{1-1}
(t^3 + t^2 + t - 1, u - t + 1) & \\
\cline{1-1}
(t^3 + t^2 + t - 1, u + t - 1) & \\
\hline
\end{array}
$$
Step (2) of Algorithm \ref{alg:rationalfunction} computes their product
$$ \prod_{v \in V_{K, \le 3}} [\LL_v(\1, T)]_3 = [(1 - T)^7(1 - T^2)^3(1 - T^3)^{10}]_3 = 1 - 7T + 18T^2 - 24T^3, $$
step (3) computes its inverse
$$ \prod_{v \in V_{K, \le 3}} [\LL_v(\1, T)^{-1}]_3 = 1 + 7T + 31T^2 + 115T^3, $$
and step (4) computes its product with $ [D(\1, T)]_3 $, which turns out to be
$$ P(\1, T) \coloneqq \left[[D(\1, T)]_3 \cdot \prod_{v \in V_{K, \le 3}} [\LL_v(\1, T)^{-1}]_3\right]_3 = 1 + 3T + 6T^2 + 12T^3. $$
Then $ [N(\1, T)]_3 = [P(\1, T)]_3 $, so it remains to compute the $ k $-th coefficients $ N_k $ of $ N(\1, T) $ for each $ k \in \{4, 5, 6\} $ via Algorithm \ref{alg:computecoefficients}, tabulated as follows.
$$
\begin{array}{|c|c|c|c|c|c|c|c|}
\hline
k & 0 & 1 & 2 & 3 & 4 & 5 & 6 \\
\hline
D_k & 1 & -4 & 3 & 0 & 0 & 0 & 0 \\
\hline
P_k & 1 & 3 & 6 & 12 & ? & ? & ? \\
\hline
M_{n - k} & - & - & - & - & 2 & 3 & 3 \\
\hline
\end{array}
$$
Since $ N_k = 9M_{n - k} $ for each $ k \in \{4, 5, 6\} $,
$$ N(\1, T) = 1 + 3T + 6T^2 + 12T^3 + 18T^4 + 27T^5 + 27T^6. $$
The complex roots of $ N(\1, T) $ have moduli $ 1 / \sqrt{3} $, so Theorem \ref{thm:weilconjectures}(5) is verified.
\end{example}

\pagebreak

\subsection{Elliptic curves}

Let $ \rho_\lambda = \rho_{E, \ell} $ be the $ \ell $-adic representation associated to an elliptic curve $ E $ over $ K $. Then
$$ \LL_v(\rho_{E, \ell}, T) =
\begin{cases}
1 - a_v(E)T^{\deg v} + q^{\deg v}T^{2\deg v} & \text{if} \ v \ \text{is good}, \\
1 - T^{\deg v} & \text{if} \ v \ \text{is split multiplicative}, \\
1 + T^{\deg v} & \text{if} \ v \ \text{is non-split multiplicative}, \\
1 & \text{if} \ v \ \text{is additive},
\end{cases}
$$
where $ a_v(E) \coloneqq 1 + p - \#E(\F_{q^{\deg v}}) $ is the trace of $ \rho_{E, \ell}(\phi_v) $. When $ \ch(K) > 3 $, there is an efficient lookup table to determine $ \epsilon(\rho_{E, \ell}) $ in terms of Jacobi symbols \cite[Theorem 3.1]{CCH05}. If $ v $ is potentially good, then
$$ \epsilon(\rho_{E, \ell}) = q^{n(\rho_{E, \ell}) - d(\rho_{E, \ell})} \cdot
\begin{cases}
1 & \text{if} \ \ord(\Delta_v(E)) \equiv 0 \mod 12, \\
(\tfrac{-3}{q})^{\deg v} & \text{if} \ \ord(\Delta_v(E)) \equiv 4, 8 \mod 12, \\
(\tfrac{-2}{q})^{\deg v} & \text{if} \ \ord(\Delta_v(E)) \equiv 3, 9 \mod 12, \\
(\tfrac{-1}{q})^{\deg v} & \text{otherwise},
\end{cases}
$$
where $ \Delta_v(E) $ is the minimal discriminant of $ E $ at $ v $. Otherwise
$$ \epsilon(\rho_{E, \ell}) = q^{n(\rho_{E, \ell}) - d(\rho_{E, \ell})} \cdot
\begin{cases}
-1 & \text{if} \ v \ \text{is split multiplicative}, \\
1 & \text{if} \ v \ \text{is non-split multiplicative}, \\
(\tfrac{-1}{q})^{\deg v} & \text{if} \ v \ \text{is additive}.
\end{cases}
$$

\begin{remark}
When $ \ch(K) \in \{2, 3\} $, it may be possible to generalise the root number results by Kobayashi \cite{Kob02} and Dokchitser--Dokchitser \cite{DD08} to $ K $.
\end{remark}

If $ E $ is \emph{constant}, in the sense that it arises as the base change to $ K $ of an elliptic curve $ E_0 $ over $ \F_q $, then the remaining are
$$  D(\rho_{E, \ell}, T) = N(E_0, T) \cdot N(E_0, qT), \ n(\rho_{E, \ell}) = 4g(C), \ w(\rho_{E, \ell}) = 1, \ c(\rho_{E, \ell}) = \id, $$
where $ N(E_0, T) $ is the numerator of the $ \zeta $-function of $ E_0 $. Otherwise,
$$ D(\rho_{E, \ell}, T) = 1, \ n(\rho_{E, \ell}) = \deg\ff(\rho_{E, \ell}) + 4g(C) - 4, \ w(\rho_{E, \ell}) = 1, \ c(\rho_{E, \ell}) = \id, $$
where $ \ff(\rho_{E, \ell}) $ can be determined from Tate's algorithm.

\begin{remark}
When $ E $ is non-constant, its associated elliptic surface $ \EE $ over $ \F_q $ has at least one singular fibre. Then $ \LL(\rho_{E, \ell}, T) = N(\rho_{E, \ell}, T) $ is the characteristic polynomial of $ \varphi_q $ acting on the orthogonal complement of the trivial sublattice of the N\'eron--Severi group $ \NS(\EE) $ under its cycle class map to the second $ \ell $-adic cohomology group $ H_\et^2(\overline{\EE}, \Q_\ell(1)) $ \cite[Theorem 4]{Shi92}. This is the idea behind Magma's implementation of \texttt{LFunction} for non-constant elliptic curves over $ \F_q(t) $.
\end{remark}

For simplicity, the following example considers $ K = \F_q(t) $, but the same algorithm will work for general $ K $ assuming a uniform algorithm to compute $ a_v(E) $.

\begin{example}
Let $ E $ be the non-constant elliptic curve over $ K = \F_7(t) $ given by $ y^2 + txy = x^3 + (t^2 + 2) $. Tate's algorithm gives
$$ \ff(\rho_{E, \ell}) = (1 / t) \cdot (t + 3) \cdot (t + 4) \cdot (t^2 + 2) \cdot (t^2 + 2t + 3) \cdot (t^2 + 5t + 3), $$
with $ (t + 3) $ and $ (t + 4) $ being non-split multiplicative and the rest being split multiplicative, so $ n(\rho_{E, \ell}) = 5 $ and $ \epsilon(\rho_{E, \ell}) = 16807 $. By Corollary \ref{cor:functionalequation}, it suffices to compute $ \LL_v(\rho_{E, \ell}, T) $ for all places $ v \in V_{K, \le n_1} = V_{K, \le 2} $, tabulated as follows.

\pagebreak

$$
\begin{array}[t]{|c|c|c|}
\hline
v & a_v(E) & \LL_v(\rho_{E, \ell}, T) \\
\hline
(1 / t) & 1 & 1 - T \\
\hline
(t) & -1 & 1 + T + 7T^2 \\
\hline
(t + 1) & 3 & 1 - 3T + 7T^2 \\
\hline
(t + 2) & -2 & 1 + 2T + 7T^2 \\
\hline
(t + 3) & -1 & 1 + T \\
\hline
(t + 4) & -1 & 1 + T \\
\hline
(t + 5) & -2 & 1 + 2T + 7T^2 \\
\hline
(t + 6) & 3 & 1 - 3T + 7T^2 \\
\hline
(t^2 + 1) & -10 & 1 + 10T^2 + 7T^4 \\
\hline
(t^2 + 2) & 1 & 1 - T^2 \\
\hline
(t^2 + 4) & -14 & 1 + 14T^2 + 7T^4 \\
\hline
(t^2 + t + 3) & 4 & 1 - 4T^2 + 7T^4 \\
\hline
(t^2 + t + 4) & 3 & 1 - 3T^2 + 7T^4 \\
\hline
(t^2 + t + 6) & 5 & 1 - 5T^2 + 7T^4 \\
\hline
\end{array}
\quad
\begin{array}[t]{|c|c|c|}
\hline
v & a_v(E) & \LL_v(\rho_{E, \ell}, T) \\
\hline
(t^2 + 2t + 2) & 12 & 1 - 12T^2 + 7T^4 \\
\hline
(t^2 + 2t + 3) & 1 & 1 - T^2 \\
\hline
(t^2 + 2t + 5) & 4 & 1 - 4T^2 + 7T^4 \\
\hline
(t^2 + 3t + 1) & 12 & 1 - 12T^2 + 7T^4 \\
\hline
(t^2 + 3t + 5) & 2 & 1 - 2T^2 + 7T^4 \\
\hline
(t^2 + 3t + 6) & 3 & 1 - 3T^2 + 7T^4 \\
\hline
(t^2 + 4t + 1) & 12 & 1 - 12T^2 + 7T^4 \\
\hline
(t^2 + 4t + 5) & 2 & 1 - 2T^2 + 7T^4 \\
\hline
(t^2 + 4t + 6) & 3 & 1 - 3T^2 + 7T^4 \\
\hline
(t^2 + 5t + 2) & 12 & 1 - 12T^2 + 7T^4 \\
\hline
(t^2 + 5t + 3) & 1 & 1 - T^2 \\
\hline
(t^2 + 5t + 5) & 4 & 1 - 4T^2 + 7T^4 \\
\hline
(t^2 + 6t + 3) & 4 & 1 - 4T^2 + 7T^4 \\
\hline
(t^2 + 6t + 4) & 3 & 1 - 3T^2 + 7T^4 \\
\hline
(t^2 + 6t + 6) & 5 & 1 - 5T^2 + 7T^4 \\
\hline
\end{array}
$$
Steps (2) to (4) of Algorithm \ref{alg:rationalfunction} compute
$$ P(\rho_{E, \ell}, T) \coloneqq \left[[D(\rho_{E, \ell}, T)]_2 \cdot \prod_{v \in V_{K, \le 2}} [\LL_v(\rho_{E, \ell}, T)^{-1}]_2\right]_2 = 1 + 49T^2, $$
Then $ [N(\rho_{E, \ell}, T)]_2 = [P(\rho_{E, \ell}, T)]_2 $, so it remains to compute the $ k $-th coefficients $ N_k $ of $ N(\rho_{E, \ell}, T) $ for each $ k \in \{3, 4, 5\} $ via Algorithm \ref{alg:computecoefficients}, tabulated as follows.
$$
\begin{array}{|c|c|c|c|c|c|c|}
\hline
k & 0 & 1 & 2 & 3 & 4 & 5 \\
\hline
D_k & 1 & 0 & 0 & 0 & 0 & 0 \\
\hline
P_k & 1 & 0 & 49 & ? & ? & ? \\
\hline
M_{n - k} & - & - & - & 1 / 49 & 0 & 1 \\
\hline
\end{array}
$$
Since $ N_k = 16807M_{n - k} $ for each $ k \in \{3, 4, 5\} $,
$$ N(\rho_{E, \ell}, T) = 1 + 49T^2 + 343T^3 + 16807T^5. $$
The complex roots of $ N(\rho_{E, \ell}, T) $ have moduli $ 1 / 7 $, so Theorem \ref{thm:weilconjectures}(5) is verified.
\end{example}

\subsection{Dirichlet characters}

Let $ \rho_\lambda = \chi $ be a primitive Dirichlet character of $ K = \F_q(t) $ of some monic non-constant square-free modulus $ P(t) \in \F_q[t] $. This is a homomorphism $ \chi : (\F_q[t] / P(t))^\times \to \C^\times $ that factors through a separable geometric extension of $ K $, so it makes sense to evaluate $ \chi $ at a finite place $ v \in V_K $ by considering any generator of the ideal of $ v $. By viewing $ \chi $ as a character of $ G_K $, it also makes sense to evaluate $ \chi $ at the unique degree one infinite place $ 1 / t \in V_K $ and compute $ \ff(\chi) = P(t) \cdot (1 / t)^{\alpha(\chi_{1 / t})} $, with \cite[Lemma 2.1]{CLDLL22}
$$ \chi(1 / t) =
\begin{cases}
1 & \text{if} \ \chi|_{\F_q^\times} = 1, \\
0 & \text{otherwise},
\end{cases}
\qquad \alpha(\chi_{1 / t}) =
\begin{cases}
0 & \text{if} \ \chi|_{\F_q^\times} = 1, \\
1 & \text{otherwise}.
\end{cases}
$$

\begin{remark}
These statements can be generalised to primitive Dirichlet characters of $ \F_q(t) $ of arbitrary modulus by studying the ramification behaviour of $ 1 / t $ over Carlitz extensions, which will be explored in a future work.
\end{remark}

\pagebreak

Then $ \LL_v(\chi, T) = 1 - \chi(v)T^{\deg v} $ for any place $ v \in V_K $, and the invariants are
$$ D(\chi, T) = 1, \ n(\chi) = \deg\ff(\chi) - 2, \ w(\chi) = 0, \ c(\chi) = \cc. $$
Now $ \epsilon(\chi) $ can be computed via Gauss sums \cite[Corollary 2.4]{DFL22} or the canonical divisor class \cite[Theorem 6.4.6]{VS06}, but this is not necessary given Corollary \ref{cor:computeepsilon}.

\begin{example}
Let $ \chi $ be the primitive Dirichlet character of $ K = \F_3(t) $ of modulus $ t^6 + t^4 + t^2 + 1 = (t^2 - t - 1)(t^2 + 1)(t^2 + t - 1) $ given by the product of
$$ \chi_{t^2 - t - 1} : t \mapsto \zeta_8, \qquad \chi_{t^2 + 1} : t + 1 \mapsto \imath, \qquad \chi_{t^2 + t - 1} : t \mapsto -1. $$
Then $ \chi(1) = \chi(2) = 1 $, so $ n(\chi) = 5 $. By Corollary \ref{cor:computeepsilon}, it suffices to compute $ \LL_v(\chi, T) $ for all places $ v \in V_{K, \le n_1} = V_{K, \le 3} $, tabulated as follows.
$$
\begin{array}[t]{|c|c|c|c|c|}
\hline
v & \chi_{t^2 - t - 1}(v) & \chi_{t^2 + 1}(v) & \chi_{t^2 + t - 1}(v) & \LL_v(\chi, T) \\
\hline
(1 / t) & 0 & 1 & 1 & 1 \\
\hline
(t - 1) & -\zeta_8^3 & -\imath & 1 & 1 + \zeta_8T \\
\hline
(t) & \zeta_8 & -1 & -1 & 1 - \zeta_8T \\
\hline
(t + 1) & \imath & \imath & -1 & 1 - T \\
\hline
(t^2 - t - 1) & 0 & -\imath & -1 & 1 \\
\hline
(t^2 + 1) & -\zeta_8^3 & 0 & -1 & 1 \\
\hline
(t^2 + t - 1) & -\zeta_8 & \imath & 0 & 1 \\
\hline
(t^3 - t^2 - t - 1) & -1 & -1 & -1 & 1 + T^3 \\
\hline
(t^3 - t^2 + 1) & \imath & \imath & 1 & 1 + T^3 \\
\hline
(t^3 - t^2 + t + 1) & \zeta_8^3 & 1 & 1 & 1 - \zeta_8^3T^3 \\
\hline
(t^3 - t - 1) & \zeta_8 & -\imath & -1 & 1 - \zeta_8^3T^3 \\
\hline
(t^3 - t + 1) & -\zeta_8^3 & \imath & -1 & 1 + \zeta_8T^3 \\
\hline
(t^3 + t^2 - t + 1) & -\zeta_8 & -1 & 1 & 1 - \zeta_8T^3 \\
\hline
(t^3 + t^2 - 1) & 1 & -\imath & 1 & 1 + \imath T^3 \\
\hline
(t^3 + t^2 + t - 1) & \imath & 1 & -1 & 1 + \imath T^3 \\
\hline
\end{array}
$$
Steps (2) to (4) of Algorithm \ref{alg:rationalfunction} compute
$$ P(\chi, T) \coloneqq \left[[D(\chi, T)]_3 \cdot \prod_{v \in V_{K, \le 3}} [\LL_v(\chi, T)^{-1}]_3\right]_3 = 1 + T + (\imath + 1)T^2 + (2\zeta_8^3 - \imath - 1)T^3, $$
Then $ [N(\chi, T)]_3 = [P(\chi, T)]_3 $, so it remains to compute the $ k $-th coefficients $ N_k $ of $ N(\chi, T) $ for each $ k \in \{4, 5\} $ via Algorithm \ref{alg:computecoefficients}, tabulated as follows.
$$
\begin{array}{|c|c|c|c|c|c|c|}
\hline
k & 0 & 1 & 2 & 3 & 4 & 5 \\
\hline
D_k & 1 & 0 & 0 & 0 & 0 & 0 \\
\hline
P_k & 1 & 1 & \imath + 1 & 2\zeta_8^3 - \imath - 1 & ? & ? \\
\hline
M_{n - k} & - & - & - & (\imath + 1)^{\cc} / 9 & 1 / 3 & 1 \\
\hline
\end{array}
$$
Now $ \epsilon(\chi) $ is not known, so this also computes $ M_2 = (\imath + 1)^{\cc} / 9 \ne 0 $, and hence
$$ \epsilon(\chi) = \dfrac{2\zeta_8^3 - \imath - 1}{(\imath + 1)^{\cc} / 9} = 9\zeta_8^3 - 9\imath - 9\zeta_8. $$
Since $ N_k = (9\zeta_8^3 - 9\imath - 9\zeta_8)M_{n - k} $ for each $ k \in \{3, 4, 5\} $,
$$ N(\chi, T) = 1 + T + (\imath + 1)T^2 + (2\zeta_8^3 - \imath - 1)T^3 - (3\zeta_8^3 - 3\imath - 3\zeta_8)T^4 + (9\zeta_8^3 - 9\imath - 9\zeta_8)T^5. $$
The complex roots of $ N(\chi, T) $ have moduli $ 1 / 3 $, so Theorem \ref{thm:weilconjectures}(5) is verified.
\end{example}

\pagebreak

\subsection{Other examples}

The range of examples of $ \rho_\lambda $ illustrated here is only limited by the computations of $ \ff(\rho_\lambda) $ and $ \LL_v(\rho_\lambda, T) $ that are easily done by hand or currently implemented in Magma. With a bit more work, many other examples of $ \LL(\rho_\lambda, T) $ and $ \epsilon(\rho_\lambda) $ can be computed from well-known results in the existing literature.

An Artin representation of $ K $ can be given as an explicit matrix representation $ \varrho : \Gal(L / K) \to \GL(V) $ for some finite extension $ L $ of $ K $, in which case $ \ff(\varrho) $ is immediate. Computing $ \LL_v(\varrho, T) $ then reduces to identifying $ \phi_v $ as elements of $ \Gal(L / K) $, which are effectively decidable thanks to the work of Dokchitser--Dokchitser \cite{DD13}. Furthermore $ w(\varrho) = 0 $, and if $ \varrho $ is irreducible and ramified somewhere, then $ D(\varrho, T) = 1 $ and $ n(\varrho) = \deg\ff(\varrho) + 4g(C) - 4 $. In many cases $ c(\varrho) \in \{\id, \cc\} $, but this is not a necessary restriction, and the case when $ c(\varrho) : F \to F $ is an arbitrary field endomorphism will be explored in a future work.

The Jacobian $ J $ of a curve $ X $ over $ K $ has $ w(\rho_{J, \ell}) = 1 $ and $ c(\rho_{J, \ell}) = \id $, and $ n(\rho_{J, \ell}) - d(\rho_{J, \ell}) = \deg\ff(\rho_{J, \ell}) + 2g(X)(2g(C) - 2) $. If $ X $ is given by an explicit hyperelliptic equation with $ \ch(K) \ne 2 $, then $ \ff(\rho_{J, \ell}) $ can be determined completely from the machinery of cluster pictures of Dokchitser--Dokchitser--Maistret--Morgan \cite{DDMM23}, with examples of families over number fields worked out explicitly \cite{ACIKMM24}. Computing $ \LL_v(\rho_{J, \ell}, T) $ could involve decomposing $ \rho_{J, \ell} $ into simpler components whose local Euler factors are computable, similar to the work of Maistret--Sutherland for many genus two curves over $ \Q $ \cite{MS25}. Note that when $ X $ is superelliptic, there are analogous decompositions of $ \rho_{J, \ell} $ in terms of cluster pictures \cite{PV23}.

By Lafforgue's theorem, the L-function $ L(\rho, s) $ of a cuspidal automorphic representation $ \rho $ of $ K $ corresponding to some $ \lambda $-adic representation $ \rho_\lambda $ of $ K $ would naturally inherit the rationality and functional equation of $ L(\rho_\lambda, s) $, so its computation can be done similarly. This may include L-functions of automorphic origin, such as those of Drinfeld modular forms arising from non-constant elliptic curves \cite{JL70}, which are certain rigid-analytic functions on the Drinfeld upper half plane constructed from Drinfeld modules \cite{Gos80}. It may be interesting to see if the positive characteristic Goss L-functions of Drinfeld modules, and more generally of $ t $-motives \cite{Gos92}, can fit in the abstract framework of the previous section.

Finally, let $ \rho_\lambda $ and $ \sigma_\lambda $ be two $ \lambda $-adic representations of $ K $. In many cases, including when $ (\ff(\rho_\lambda), \ff(\sigma_\lambda)) = 1 $, the invariants of $ \rho_\lambda \otimes \sigma_\lambda $ can be expressed in terms of those of $ \rho_\lambda $ and $ \sigma_\lambda $. In particular, having an explicit expression of $ \epsilon(\rho_\lambda \otimes \sigma_\lambda) $ in terms of $ \epsilon(\rho_\lambda) $ and $ \epsilon(\sigma_\lambda) $ via a global version of Proposition \ref{prop:unramifiedepsilon} caps the worst running time in the computation of $ \LL(\rho_\lambda \otimes \sigma_\lambda, T) $ in terms of those for $ \LL(\rho_\lambda, T) $ and for $ \LL(\sigma_\lambda, T) $. This will be explored in a future work.

\section*{Acknowledgments}

Much of this paper is part of my PhD thesis, and I thank Vladimir Dokchitser for his guidance and support throughout. I also thank Richard Griffon, Timo Keller, Douglas Ulmer, and Hanneke Wiersema for their comments, as well as Tim Dokchitser and C\'eline Maistret for their encouragements. During the research into this paper, I was supported by the Engineering and Physical Sciences Research Council [EP/S021590/1], EPSRC Centre for Doctoral Training in Geometry and Number Theory (London School of Geometry and Number Theory), University College London. During the writing of this paper, I was part of the Scalable Theorem Proving via Mathematical Databases project funded by the AI for Math Fund managed by Renaissance Philanthropy in partnership with founding donor XTX Markets.

\pagebreak

\bibliographystyle{abbrv}
\bibliography{main}

\end{document}